\documentclass[11pt, oneside]{amsart}       
\usepackage{geometry}                        
\usepackage[parfill]{parskip}            
\usepackage{graphicx}                
                            \usepackage{amsthm}    
\newtheorem{theorem}{Theorem}[section]
\newtheorem{conjecture}{Conjecture}[section]
\newtheorem{corollary}{Corollary}[section]
\newtheorem{definition}{Definition}[section]
\newtheorem{lemma}[theorem]{Lemma}
\theoremstyle{remark}
\newtheorem*{remark}{Remark}
\usepackage{tcolorbox}
\newtheorem*{example}{Example}

\title{A note on Semi-Infinite Non-Commutative Hodge Theory and LG-Models}
\author{EMILE BOUAZIZ}

\begin{document}\begin{abstract} We study some semi-infinite invariants associated to Landau-Ginzburg models. These specialize classically to the usual twisted de Rham package and in the case of vanishing potential to the chiral de Rham complex of Malikov, Schechtman and Vaintrob, \cite{MSV}. Further we offer some small evidence that these semi-infinite invariants can be associated to dg- categories more generally, so that in the classical limit they reproduce the usual non-commutative Hodge theory associated to the category. \end{abstract}
\maketitle
\section{acknowledgements} I've benefited from numerous useful conversations whilst this work was done. I'd like in particular to thank Vassili Gorbounov, Ian Grojnowski, Owen Gwilliam and Nitu Kitchloo. Further I'd like to thank MPIM Bonn and all the staff there for providing a wonderful working environment.
\section{introduction} In the paper \cite{MSV}, the authors introduce a novel infinite dimensional extension of the complex of differential forms associated to a smooth $\mathbb{C}$-variety $X$. They call it the \emph{Chiral de Rham Complex}, and denote it $\Omega^{ch}_{X,dR}$. This complex has the structure of a differential graded vertex algebra with differential $d^{ch}_{dR}$. It is endowed with an additional \emph{conformal} grading by non-negative integers, such that the weight $0$ subspace agrees with the classical de Rham complex. This sheaf of vertex algebras can be viewed as a non-linear analogue of the $bc$-$\beta\gamma$ system, familiar to physicists. The authors construct $\Omega^{ch}_{X,dR}$ via an etale gluing procedure after writing it down for $\mathbb{A}^{d}$, in which case it is simply a tensor product of $bc$-$\beta\gamma$ systems. NB that an alternative construction in terms of $D$-modules on the algebraic loop space of $X$ was given by Kapranov and Vasserot in \cite{KV}. \\ \\ As well as $\Omega^{ch}_{X,dR}$, one may define analogues of the sheaf of polyvectors and the Hodge complex (the de Rham complex with deleted differential) on $X$. We denote these respectively by $\Theta^{ch}_{X}$ and $\Omega^{ch}_{X}$. $\Theta^{ch}_{X}$ is a  differential graded vertex algebra (dgVA henceforth) and it acts on $\Omega^{ch}_{X}$. Further this module admits an extra differential coming from $d^{ch}_{dR}$. This is simply to say that there is an enhancement of the usual calculus associated to $X$ to a semi-infinite version (we use this term only suggestively) thereof. Assuming for simplicity that $X$ is affine with $\mathcal{O}(X)$=$A$, the action of polyvectors on differential forms can be obtained from the Hochschild (co)homology package associated to $A$, as can the de Rham differential. This is simply the famous HKR isomorphism. Crucially, this package exists for arbitrary associative (differential graded) $\mathbb{C}$-algebras. According to Kontsevich (\cite{Ko}) we should view it as a non-commutative version of Hodge theory. It is the goal of this note to provide some simple examples in which this nc-Hodge package can be enhanced to a semi-infinite version of such. \\ \\ If $(X,f)$ is a Landau-Ginzburg model, then there is attached a $\mathbb{Z}/2$-graded dg-category of matrix factorizations, $\mathbf{MF}(X,f)$. The nc-Hodge package associated to this is by now well known to correspond to twisting the usual complexes of forms or polyvectors by the action of the element $df$. As such, and owing to beautiful work of Sabbah (\cite{Sa})), it is closely related to vanishing cycles cohomology of the pair $(X,f)$. We'll see that in this case a chiral enhancement exists, and moreover that we can say a fair deal about it. In particular we'll describe the representation theory of the vertex algebra we construct as well as proving a finiteness result allowing us to construct an enhancement of the vanishing cycles euler characteristic, $\chi_{van}(f)\in\mathbb{Z}$, to a refined version $\chi^{ch}_{van}(f)\in\mathbb{Z}[[q]].$

 \section{Recollections} \subsection{Basics of vertex algebra theory}We include here some background on the theory of vertex algebras, largely in order to fix some notation. Our main references throughout will be the books of Kac, \cite{Kac}, and of Ben-Zvi and Frenkel, \cite{BZF}. \begin{definition} A differential graded vertex algebra or dgVA is a tuple $(V,\partial_{V},T,\Omega,Y)$ where: \begin{itemize} \item $V$ a dg- vector space, with cohomological differential $\partial_{V}$.  \item $T$ an endomorphism of $V$, referred to as the \emph{infinitesimal translation}. Note that $T$ is assumed an endomorphism of $V$ considered as a dg- vector space, so that it is of cohomological degree $0$ and $[T,\partial_{V}]=0$.\item $\Omega$ a cycle of cohomological degree $0$, referred to as the \emph{vacuum vector} of $V$. \item $Y:V\otimes V\longrightarrow V((z))$ a multiplication map. We will package this as associating to any $a\in V$ a \emph{field} $a(z)\in End(V)[[z^{-1},z]]$.  We can think of this as a family of endomorphisms, $a_{(n)}$, so that $a(z)=\sum_{n}a_{(n)}z^{n}$. The reader is cautioned that this differs from standard indexing conventions.\end{itemize} These data are further assumed to satisfy the following conditions; \begin{itemize} \item The field associated to the vacuum vector is the identity. \item $T$ kills the vacuum vector, i.e. $T(\Omega)=0$. \item For all $v\in V$, $v(z)\Omega = v+\mathcal{O}(z)$.\item For all vectors $a$ we have $[T,a(z)]=\partial_{z}a(z)$. \item For all $a,b\in V$, the fields $a(z)$ and $b(w)$ are mutually \emph{local}. That is to say for $N>>0$ we have $(z-w)^{N}[a(z),b(w)]=0$.\item The differential $\partial_{V}$ acts as a \emph{derivation} of $V$ in the sense that we have $[\partial_{V},a(z)]=(\partial_{V}a)(z)$ for all vectors $a$. \end{itemize}  \end{definition} The vertex algebras with which we will deal in this note will all be equipped with an additional grading by non-negative integers. We will refer to such gradings as \emph{conformal}. \begin{definition} Let $V$ be a vertex algebra. A grading $V=\bigoplus_{q\geq 0}V^{(q)}$ is said to be a \emph{conformal} grading if the following conditions are satisfied;\begin{itemize}\item The vacuum vector $\Omega$ is of degree $0$. \item The infinitesimal translation $T$ is of degree $1$. \item If $a\in V^{(q)}$ then the endomorphism $a_{(n)}$ is of degree $q+n$.\end{itemize}\end{definition} We will not include the definition of a module for a dgVA here, beyond to say that it is a dg- vector space $M$ equipped with an action map $V\otimes M\longrightarrow M((z))$ satisfying conditions guaranteeing for example that $V$ is naturally a module over itself. There is also a notion of a conformally graded module for a conformally graded vertex algebra. We refer the reader to \cite{BZF} for further details.

 \subsection{Examples of Vertex Algebras}

 We now sketch some basic examples of dgVAs. We let $\mathcal{H}$ be the infinite dimensional Heisenberg algebra generated by elements $\{x_{i},y_{j}\}_{i,j\in\mathbb{Z}}$ with commutation relations $$[y_{i},x_{j}]=\delta_{i+j,0}.$$ There is an abelian Lie subalgebra $\mathcal{H}^{+}$ generated by elements $\{x_{<0},y_{\leq 0}\}$. Let us denote by $V$ the induction of the trivial $\mathcal{H}^{+}$-module to all of $\mathcal{H}$. It is well known that this induced module obtains the structure of a vertex algebra. Note that the underlying vector space of this vertex algebra is isomorphic to $$sym_{\mathbb{C}}\{x_{\geq0},y_{> 0}\}=\mathbb{C}\big[x_{i},y_{j+1}\big]_{i,j\geq0}.$$ We refer the reader to \cite{BZF} for an elucidation of the fields. There is an odd version of this. Consider the infinite dimensional Clifford algebra $Cl$ generated by elements $\{\phi_{i}\}_{i\in\mathbb{Z}}$ of cohomological degree +1 and elements $\{\psi_{i}\}_{i\in\mathbb{Z}}$ of degree -1 with the (super-) commutation relations $$[\psi_{i},\phi_{j}]=\delta_{i+j,0}$$ As in the even case there is an abelian sub-algebra generated by $\{\phi_{<0},\psi_{\leq 0}\}$ and we can induce the trivial module for this to all of $Cl$. We denote the resulting module $\bigwedge$ and recall that it also admits the structure of a dgVA. We further define $V_{d}:=V^{\otimes d}$ and $\bigwedge_{d}:=\bigwedge^{\otimes d}$. Finally we define $\bigwedge^{!}$ in a manner analogous to $\bigwedge$, except we induce from the subalgebra  $\{\phi_{\leq0},\psi_{<0}\}$
 \begin{definition} \begin{itemize}\item We define the dgVA $\Omega^{ch}_{\mathbb{A}^{d}}$ to be the tensor product algebra $V_{d}\bigotimes\bigwedge_{d}$. This is referred to as the sheaf of chiral differential forms on the affine $d$-space $\mathbb{A}^{d}$. \item We define the dgVA $\Theta^{ch}_{\mathbb{A}^{d}}$ to be the tensor product algebra $V_{d}\bigotimes\bigwedge^{!}_{d}$. This is referred to as the sheaf of polyvector fields on the affine $d$-space $\mathbb{A}^{d}$.\item The chiral de Rham differential, $d^{ch}_{dR,\mathbb{A}^{d}}$ is the derivation of  $\Omega^{ch}_{\mathbb{A}^{d}}$ defined in operatorial terms as $\sum_{i\in\mathbb{Z}}\sum_{j=1,...,d}y^{j}_{i}\phi^{j}_{-i}.$ This can be checked to be a derivation of the vertex algebra (\cite{MSV}). The resulting dgVA is denoted $\Omega^{ch}_{dR,\mathbb{A}^{d}}.$  \end{itemize}\end{definition} The following theorem underlies the construction of the chiral de Rham complex associated to an arbitrary smooth variety $X$. Before stating it we note that the above three vertex algebras can be completed $(x^{1}_{0},...,x^{d}_{0})$-adically, obtaining  $\widehat{\Omega}^{ch}_{\mathbb{A}^{d}}$ etc. We think of these as vertex algebras of the formal $d$-disc $\Delta_{d}$. We let  $G_{d}$ be the pro-unipotent group of automorphisms of this formal scheme. The following theorem is proven in \cite{MSV}. \begin{theorem} The action of $G_{d}$ extends naturally to an action on the vertex algebras  $\widehat{\Omega}^{ch}_{\mathbb{A}^{d}}$,  $\widehat{\Theta}^{ch}_{\mathbb{A}^{d}}$ and  $\widehat{\Omega}^{ch}_{dR,\mathbb{A}^{d}}$.\end{theorem} \begin{proof} Consult the original paper \cite{MSV}. \end{proof} As observed in \cite{MSV}, this theorem allows us to globalize the constructions from a disc $\Delta_{d}$ to any smooth $d$-dimensional variety via the method of \emph{Gelfand}-\emph{Kazhdan} \emph{formal geometry}. The resulting sheaves of vertex algebras will be denoted $\Omega^{ch}_{X}$ etc. There is an evident conformal grading on each of them, with $x_{i}$ of weight $i$ and so on.

 \subsection{Hodge Theory of Landau-Ginzburg Models} We will now briefly recall the basics of Hodge theory for a Landau-Ginzburg model. Note that according to work of numerous authors (we mention here A. Efimov, \cite{Efimov},and A. Preygel, \cite{Prey}), this is equivalent to the nc-Hodge theory of the matrix factorization category associated to $(X,f)$. A word of caution; there are more sophisticated notions of Hodge theory for LG-models that we will not be touching on in this note, as it seems unlikely that such things are obtainable from nc-Hodge theory. The interested reader is referred to the numerous brilliant works of Sabbah and Mochizuki on \emph{exponential mixed Hodge modules}.  \begin{definition} For an LG-model $(X,f)$ we define; \begin{itemize}\item  $\Theta_{f}:=\Big(\bigwedge^{-*}\Theta_{X},\iota_{df}\Big)$  where $\Theta_{X}$ denotes the sheaf of polyvector fields on $X$ and $\iota$ denotes contraction. $\Theta_{f}$ will be referred to as the sheaf of polyvectos on $(X,f)$.\item $\Omega_{f}:=\Big(\bigwedge^{+*}\Omega_{X},df\wedge\Big).$ This will be referred to as the sheaf of differential forms on $(X,f)$. \item $\Omega_{dR,f}:=\Big(\bigwedge^{+*}\Omega_{X},d_{dR}+df\wedge\Big).$ This will be referred to as the de Rham complex on $(X,f)$\end{itemize}\end{definition} \begin{remark} We note here that $\Theta_{f}$ is a (commutative) differential graded algebra with a module $\Omega_{f}$, and further that this module admits an extra differential. Note also that $\Omega_{f}$ \emph{is not} a differential graded algebra unlike in the case of vanishing potential. \end{remark}\begin{example}An instructive simple example is gotten by taking $(X,f)$ to be $(\mathbb{A}^{1},z^{d+1})$. Here $\Theta_{f}\cong\mathbb{C}[z]/z^{d}$, the sheaf of differential forms is isomorphic to $\big(z^{d}:\mathbb{C}[z]\longrightarrow\mathbb{C}[z]\big)$ and the natural Hodge- de Rham spectral sequence degenerates at the first page. Further, the euler characterstic of the de Rham complex is -$d$ which is -$1$ multiplied by the euler characterstic of the vanishing cycles euler characteristic of $f$, in agreement with the theorem of Sabbah (\cite{Sa})\end{example}

\begin{remark} In the sequel we will freely use the language of \emph{derived schemes} as it will simplify some proofs. This doesn't require anything like the full force of Derived Algebraic Geometry so should not cause the reader great difficulties. With this in mind let us now note that the sheaf of commutative differential algebras $\Theta_{f}$ is the sheaf of functions on the \emph{derived critical locus} of $f$, denoted here by $T^{*}_{df}[-1]X$. These objects have recieved a good deal of study recently owing to their status as local Darboux models for $-1$-symplectic varieties (cf citations). \end{remark}

\section{The Vertex Algebra associated to a Landau-Ginzburg Model}\subsection{BRST Reductions of Vertex Algebras}Iin this section we will introduce the basic objects of study of this note. They will be constructed from known objects via a general construction in the theory of vertex algebras, known as \emph{BRST reduction}. This is a procedure for modifying a given dgVA by deforming the differential. \\ \\ Now, let $a\in V$ be a vector, which we assume to be a cycle of cohomological degree $1$. \\ \\We see immediately from the relation $[T,a(z)]=\partial_{z}a(z)$ that $$[T,a_{-1}]=Res_{z=0}(\partial_{z}a(z))=0.$$ It follows further from \emph{associativity of the operator product expansion} (cf. \cite{Kac}) that $$[a_{-1},b(z)]=(a_{-1}b)(z)$$ for all vectors $b$. This is to say that $a_{-1}$ acts a derivation of $V$. We observe that this derivation has cohomoligical degree $0$ as $a$ was assumed to have cohomological degree $1$. Putting all of this together we deduce the following \begin{lemma} If $a\in V$ is as above and further if $a_{(-1)}^{2}=0$, then we have a dgVA with underlying dg- vector space $(V,\partial_{V}+a_{(-1)})$ and vacuum vector and fields unchanged from those of $V$. Further, if $V$ admits a conformal grading with $a\in V^{(1)}$, then the associated vertex algebra inherits this grading. \end{lemma}\begin{proof} Nothing remains to be checked. \end{proof} \begin{definition} Let $a\in V$ be as in the lemma above, then the vertex algebra so constructed is referred to as the \emph{BRST reduction} of $V$ by $a$. It is denoted $V^{BRST}_{a}$.\end{definition} \begin{example} Let $V=\Omega^{ch}_{\mathbb{A}^{d}}$. Further let $$a:=\sum_{j=1,...,d}y^{j}_{1}\phi^{j}_{0}\in\Omega^{ch}_{\mathbb{A}^{d}}.$$ This satisfies all the requisite conditions and thus we can perform BRST reduction. It is easy to see that we obtain  $V=\Omega^{ch}_{dR,\mathbb{A}^{d}}$.\end{example}
\subsection{The Construction} Let $(X,f)$ be our Landau-Ginzburg model. We'll construct a conformally graded dgVA $\Theta^{ch}_{f}$ with conformal weight $0$ subspace isomorphic to $\Theta_{f}$ via BRST reduction. Note, there is an element $df\in\Theta^{ch}_{f}$. In etale local coordinates we have  $$df=\sum_{j}\partial_{j}f(x^{1}_{0},...,x^{d}_{0})\phi^{j}_{1}.$$ Observe that this is of cohomological degree $+1$, as well as conformal degree $+1$. Finally we can check that we have $\Big(Res_{z=0}(df)(z)\Big)^2=0$. It follows that the pair $(\Theta^{ch}_{f},df)$ satisfies all the conditions needed to perform BRST reduction and we can thus define; \begin{definition} The conformally graded dgVA $\Theta^{ch}_{f}$ is by definition the BRST reduction $\Theta^{ch,BRST}_{X,df}.$ \end{definition} Unsurprisingly, there are analogues of $\Omega_{f}$ and $\Omega_{dR,(X,f)}$ as well. We package the following definitions into a trivial lemma; \begin{lemma} Let $\Omega^{ch}_{f}$ denote the dg- vector space, $(\Omega^{ch}_{X},\partial^{ch}_{f}:=(df)_{0})$, where we note that we are now taking the $0$-mode of $df\in\Omega^{ch}_{X}$. Then $\Omega^{ch}_{f}$ admits a natural structure of a conformally graded vertex module for the vertex algebra. Further, this module admits an extra differential $d^{ch}_{dR}$, which is simply to say that $[\partial^{ch}_{X},d^{ch}_{dR}]=0$. \end{lemma} \begin{proof}This is a simple computation in local coordinates.\end{proof} Of course, $(\Theta^{ch}_{f},\Omega^{ch}_{f},\Omega^{ch}_{dR,f})$ is the desired semi-infinite enhancement of the nc-Hodge theory associated to $(X,f)$. We mention here that during the writing of this note we learned from V. Gorbounov of the lovely paper \cite{Go} in which a similar construction is studied (in more detail) in a particular case. \begin{remark} Observe that $\Theta^{ch}_{f}$ has a holomorphic (negative modes all act trivially) sub-algebra, locally generated by the $x$ and $\psi$ variables. This sub-algebra has a very elegant description as the algebra of functions on the space of arcs into the derived scheme $T^{*}_{df}X[-1]$. In fact, if we adopt the language of \emph{chiral differential operators} $\mathcal{D}^{ch}$ (cf. \cite{MSV}), we should be able to (make sense of and) prove an isomorphism $\mathcal{D}^{ch}_{T^{*}_{df}[-1]X}\cong\Theta^{ch}_{f}.$ This heuristic, together with a result of Malikov and Schechtman, \cite{MS}, suggests that we should be able to identify a suitably nice category of representations of this vertex algebra. \end{remark} \subsection{Conformal Vertex Modules Associated to the Landau-Ginzburg Model} Recall that $\Theta^{ch}_{f}$ admits a conformal grading, we wish to study the category $\mathbb{Z}_{\geq 0}$-modules wrt this conformal grading, note that we will also refer to the grading on the module as conformal. The answer is remarkably simple as we shall see below. \begin{lemma} Letting $\mathcal{M}_{+ve}(\Theta^{ch}_{f})$ denote the category of conformally graded modules for the vertex algebra $\Theta^{ch}_{f}$ and $\mathcal{D}^{l}(T^{*}_{df}[-1]X)$ the category of left $D$-modules on the derived critical locus. Then there is a naturally defined equivalence of categories $$\mathcal{M}_{+ve}(\Theta^{ch}_{f})\longrightarrow\mathcal{D}^{l}(T^{*}_{df}[-1]X).$$\end{lemma}\begin{proof} We break the proof into simple steps, the main input is essentially just Kashiwara's Lemma. \begin{enumerate} \item Here we construct the functor which will realize the equivalence in the case of $(X,f)=(\mathbb{A}^{1},0)$. Now a module, $\mathcal{M}$, for the vertex algebra $\Theta^{ch}_{\mathbb{A}^{1}}$ admits an action by the infinite dimensional Lie algebra generated by the elements $\{x_{i},y_{i},\phi_{i},\psi_{i}\}_{i\in\mathbb{Z}}$ with the evident commutation relations. (Abusing notation slightly, we think of these as modes of elements of the vertex algebra, so for example $y_{0}$ is identified with the residue of $y_{1}\in\Theta^{ch}_{\mathbb{A}^{1}}$.) \\ \\In particular the elements $\{x_{0},y_{0},\phi_{0},\psi_{0}\}$ generate the algebra $Diff(T^{*}[-1]\mathbb{A}^{1})$ of differential operators on the odd cotangent bundle of $\mathbb{A}^{1}$. We define the functor, $$(-)^{!}:\mathcal{M}_{+ve}(\Theta^{ch}_{\mathbb{A}^{1}})\longrightarrow\mathcal{D}^{l}(T^{*}[-1]\mathbb{A}^{1}),$$ by taking the singular vectors with respect to the negative modes. That is to say $$\mathcal{M}^{!}:=\{m\in\mathcal{M}:(x_{i},y_{i},\phi_{i},\psi_{i})m=0,\forall i<0\}.$$  Negative modes commute with $0$-modes so $\mathcal{M}^{!}$ is indeed a module for differential operators on $T^{*}[-1]\mathbb{A}^{1}$. \item We can now observe compatibility with taking products of varieties as well as etale gluings. This allows us to globalize the construction to a functor to $$(-)^{!}:\mathcal{M}_{+ve}(\Theta^{ch}_{X})\longrightarrow\mathcal{D}^{l}(T^{*}[-1]X).$$ Consideration of the differentials shows that this in fact induces a functor on the twists of both sides, ie $$(-)^{!}:\mathcal{M}_{+ve}(\Theta^{ch}_{f})\longrightarrow\mathcal{D}^{l}(T^{*}_{df}[-1]X).$$\item We show that the functor constructed in (1) above is an equivalence by explicitly constructing the inverse as an induced module. For a $D$-module $\mathcal{N}$ on $T^{*}[-1]\mathbb{A}^{1}$, write $\iota(\mathcal{N}):=\mathcal{N}[x_{i},y_{i},\phi_{i},\psi_{i}]_{i>0}$. Observe that this is naturally a vertex module over $\Theta^{ch}_{\mathbb{A}^{1}}$. Further, for $\mathcal{M}\in\mathcal{M}_{+ve}(\Theta^{ch}_{\mathbb{A}^{1}})$, there is a natural map $$\epsilon:\iota(\mathcal{M}^{!})\longrightarrow\mathcal{M}.$$ Injectivity is obvious enough, just as in Kashiwara's lemma. For surjectivity, let us take $m\in\mathcal{M}$ and assume it is of conformal weight $q$.\\ \\ Consider now $\epsilon$ to be a map of modules for $\mathbb{C}[x_{i},y_{i},\phi_{i},\psi_{i}]_{i\in[-q,q]\setminus\{0\}}$. Kashiwara's lemma implies we can find finitely many $(x_{i},y_{i},\phi_{i},\psi_{i})_{i\in[-q,0)}$-torsion elements of $\mathcal{M}$ whose span under the action of $\mathbb{C}[x_{i},y_{i},\phi_{i},\psi_{i}]_{i\in(0,q]}$ contains $m$. These elements must all manifestly be of conformal weight at most $q$. As such they are necessarily killed by all elements of conformal weight less than $-q$, and thus they are singular vectors for all negative modes, and we have proven surjectivity of $\epsilon$. \item Using the gluing construction of the functor $(-)^{!}$ and the fact that it is compatible with differentials on both sides in the presence of a potential $f$, we conclude that the functor induces the desired equivalnce.\end{enumerate}\end{proof}\begin{remark}\begin{itemize}\item The bounded below derived category of $\mathcal{D}^{l}(T^{*}_{df}[-1]X)$ is equivalent to the bounded below derived category of $D$-modules on $X$ with topological support in $\mathbf{crit}(f)$, which by definition is equivalent to $D$-modules on $\mathbf{crit(f)}$. \item Note that we can define weak equivalences of vertex modules in such a way that they are reflected by the functor $^{!}$. We take the localization with respect to these and denote the resulting category $d\mathcal{M}_{+ve}(\Theta^{ch}_{f})$.\end{itemize}\end{remark} The above remarks allow us to deduce the following theorem, \begin{tcolorbox}\begin{theorem} There is a naturally defined equivalence of dg-categories $$\mathcal{D}\mathcal{M}_{+ve}(\Theta^{ch}_{f})\longrightarrow\mathcal{D}^{l}(\mathbf{crit}(f)).$$ In particular if $f$ has an isolated singularity then $\mathcal{M}_{+ve}(\Theta^{ch}_{f})$ is just the category of vector spaces. \end{theorem}\end{tcolorbox}

\subsection{Finiteness Properties} In the presence of an interesting potential $f$, we cannot of course expect the variety $X$ to be proper. Properness of $X$ can be understood as a finiteness property for its category of perfect complexes, and taking this non-commutative point of view one obtains the correct notion of properness for a Landau-Ginzburg model. Such a definition oughtto be equivalent to homological properness of $\mathbf{MF}(X,f)$ in the sense of Kontsevich, \cite{Ko}. \begin{definition}We say that $(X,f)$ is proper if the critical locus $\mathbf{crit}(f)$ is proper. \end{definition} We would like to know that in this case the cohomology of the associated vertex algebra is finite as well, at least in the conformally graded sense. We prove this below.

\begin{tcolorbox}\begin{theorem} Let $(X,f)$ be a proper Landau-Ginzburg model. Then for every fixed conformal weight $j$, the weight $j$ component of the total hypercohomology of the sheaf of vertex algebras, $\Theta^{ch}_{f}$, is finite dimensional. That is to say for all $j$ we have $$\dim_{\mathbb{C}}\mathbb{H}^{*}\Big(X,\Theta^{ch,(j)}_{f}\Big)<\infty.$$\end{theorem}\end{tcolorbox}\begin{proof} We will introduce an increasing filtration $\mathcal{F}$ of the sheaf $\Theta^{ch,(j)}_{f}$. It is a mild modification of the filtration introduced in \cite{MSV}. We define it first for $\mathbb{A}^{1}$. Recall from above the meaning of the symbols $\{x,\psi,\phi,y\}$ so that we can make the identification of vector spaces $$\Theta^{ch}_{\mathbb{A}^{1}}=\mathbb{C}[x_{i},y_{i+1},\psi_{i},\phi_{i+1}]_{i\geq 0}.$$ We now stipulate the following, $$\{x_{0},\psi_{0}\}<x_{1}<x_{2}<...<\psi_{1}<\psi_{2}<...<\phi_{1}<\phi_{2}<...y_{1}<y_{2}<...$$ We now extend this to all of $\Theta^{ch}_{\mathbb{A}^{1}}$ lexicographically. This is of course not exhaustive but is easily seen to be exhaustive on each fixed conformal weight component. Now observe the following simple facts, \begin{itemize}\item The filtration $\mathcal{F}$ is compatible with the vertex algebra structure on $\Theta^{ch}_{\mathbb{A}^{1}}.$\item The $0$-th associated graded is the algebra of polyvectors on $\mathbb{A}^{1}$. \item For all $j>0$, $Gr^{j}_{\mathcal{F}}\Theta^{ch}_{\mathbb{A}^{1}}$ is a perfect module over $Gr^{0}_{\mathcal{F}}\Theta^{ch}_{\mathbb{A}^{1}}.$\end{itemize} As in \cite{MSV} we observe compatibility with coordinate transformations. Thus, we produce a filtration $\mathcal{F}_{X}$ on each $\Theta^{ch}_{X}$. Again, as in \cite{MSV} we note that the associated graded $Gr^{>0}_{\mathcal{F}_{X}}\Theta^{ch}_{X}$ is a quasi-coherent sheaf on the scheme $X$. Reducing to the case of $X=\mathbb{A}^{1}$ we see that, for all $j>0$,  $Gr^{j}_{\mathcal{F}_{X}}\Theta^{ch}_{X}$ is a direct sum of locally free sheaves in various cohomological degrees. We denote the resulting perfect complex on $X$ by $\mathcal{V}_{j}$.  \\ \\ Note that so far we have not mentioned the potential $f$. As before, let $\partial^{ch}_{f}$ denote the differential on $\Theta^{ch}_{f}$. Observe that  $\partial^{ch}_{f}$ takes $\psi$ variables to $x$ ones and $y$ ones to $\phi$ ones. Further it is of conformal weight $0$. These two facts together imply that it gives a well-defined filtration on  $\Theta^{ch}_{f}$. Observe that $$Gr^{0}_{\mathcal{F}_{X}}\Theta^{ch}_{X}\cong\mathcal{O}_{T^{*}[-1]X},$$ the sheaf of functions on the -$1$-shifted cotangent bundle. Examining the differential, we see that it acts as contraction by $df$ on $\mathcal{O}_{T^{*}[-1]X}$, so that $$Gr^{0}_{\mathcal{F}_{f}}(\Theta^{ch}_{f})\cong\mathcal{O}_{T^{*}_{df}[-1]X}.$$ Crucially we also note that for all $j>0$,  $Gr^{j}_{\mathcal{F}_{X}}\Theta^{ch}_{X}$ is isomorphic (as a sheaf over $T^{*}_{df}[-1]X$) to $\pi^{*}\mathcal{V}_{j}$, where $$\pi: T^{*}_{df}[-1]X\longrightarrow X$$ is the natural map. In particular each $Gr^{j}\Theta^{ch}_{f}$ is a perfect sheaf on the derived critical locus $T^{*}_{df}[-1]X$. \\ \\ In order to complete the proof, it of course suffices (by the spectral sequence associated to a filtered complex) to prove that $$\dim_{\mathbb{C}}\mathbb{H}^{*}\Big(X,Gr\Theta^{ch,(j)}_{f}\Big)<\infty.$$ With this in mind we need only recall the following basic fact from derived algebraic geometry; \begin{itemize}\item Suppose $\mathfrak{X}$ is a derived scheme with proper classical truncation $\mathfrak{X}^{cl}$, and suppose further that the total cohomology sheaf $\bigoplus\mathcal{H}^{-j}(\mathcal{O}_{\mathfrak{X}})[j]$ is coherent over $\mathfrak{X}^{cl}$. Then for any perfect sheaf $\mathfrak{F}$ on $\mathfrak{X}$, the total cohomology $\mathbb{H}^{*}(\mathfrak{X},\mathfrak{F})$ is finite dimensional.\end{itemize} In order to prove this one can use induction on the Postnikov tower of $\mathfrak{X}$, noting that the base case is a classical theorem of Serre. The condition that $\mathbf{crit}(f)$ is proper allows us to set $\mathfrak{X}:=T^{*}_{df}[-1]X$ and $\mathfrak{F}:= Gr^{j}\Theta^{ch}_{f},$ we deduce immediately from the lemma the finiteness of $$\dim_{\mathbb{C}}\mathbb{H}^{*}\Big(X,Gr\Theta^{ch,(j)}_{f}\Big),$$ and thence the theorem.\end{proof} We observe now the following - \begin{corollary} For a proper LG-model $(X,f)$, the total hypercohomology in each fixed conformal weight of $\Omega^{ch}_{f}$ is also finite. In particular we can produce a well-defined $q$-graded euler characteristic, defined as follows; \begin{tcolorbox}$$\chi^{ch}_{van}(f):=\sum\chi\Big(\mathbb{H}^{*}(X,\Omega^{ch,(j)}_{f})\Big)q^{j}\in\mathbb{Z}[[q]].$$\end{tcolorbox}\end{corollary} \begin{example} Let us take $(X,f)=(\mathbb{A}^{1},z^{d+1})$. We can then consider the associated \emph{bi}-\emph{graded} complex $\Omega^{ch}_{f}$  by considering the standard $\mathbb{G}_{m}$-action on $\mathbb{A}^{1}$. Computing the bi-graded euler characteristic is then easy, note the we write $z$ for the extra grading variable. As is standard in the literature on $q$-series we write $$\theta_{q}(z)=\prod_{n\geq0}(1-q^{n}z)(1-q^{n+1}z^{-1}).$$ We then compute the bigraded euler characteristic to be simply $$\theta_{q}\big(-z^{-d}\big)\theta_{q}(z)^{-1}=-z^{-d}\theta_{q}(z^{d})\theta_{q}(z)^{-1}.$$ Note that in the $q\longrightarrow0$ limit this produces -$z^{-d}(1+z^{1}+...+z^{d-1})$, which is consistent with $\{z^{d-1}dz,...,dz\}$ forming a basis of the first twisted de Rham cohomology group coupled with the vanishing of all other twisted de Rham cohomology groups.\end{example}

\section{Some Vague Speculation} The formalism of Hochschild (co)-homology produces, for a dg-category $\mathbf{C}$, a commutative differential graded algebra $H^{*}(\mathbf{C})$, of \emph{Hochschild cochains}, with a distinguished module $H_{*}(\mathbf{C })$ of \emph{Hochschild chains}, which module admits an additional circle action. We'll assume familiarity with this formalism. We believe that under suitable conditions on the category $\mathbf{C}$, it should be possible to produce an enhancement of this structure to a semi-infinite such. \\ \\ We will state this as a conjecture, although we will be very imprecise. \begin{conjecture} Under suitable conditions on the dg-category dgVA,  $H^{*}_{ch}(\mathbf{C})$ with a graded module  $H^{ch}_{*}(\mathbf{C})$, which admits an extra $S^{1}$-action. Further it holds that; \begin{enumerate} \item The conformal weight $0$ limit reproduces the usual Hochschild package.\item This reproduces the usual chiral de Rham package when inputted $\mathbf{Perf}(X)$.\item It reproduces the above constructed chiral de Rham package of the LG-model $(X,f)$ in the case of $\mathbf{MF}(X,f)$.\item For a proper dg-category the associated cohomology groups are finite in the conformally graded sense. \item The inclusion of the $0$-weight subspace of the derived $S^{1}$-invariants, $$H_{*}(\mathbf{C})^{S^{1}}\longrightarrow H_ {*}(\mathbf{C})^{ch,S^{1}},$$ is a weak equivalence.\end{enumerate}\end{conjecture}

\begin{remark} We should remark that in the case of $\mathbf{Perf}$, (5) is proven in the original paper \cite{MSV}, where it reduces to a very easy statement in the case of $\mathbb{A}^{1}$. It is easy to deduce (5) in the case of $\mathbf{MF}(X,f)$ from this. \end{remark}\subsection{Further Examples} We include here a brief sketch of how this can be done when $\mathbf{C}$ is the category of representations of a finite dimensional Lie algebra $\mathfrak{g}$.\\ \\ We denote the underlying vector space of $\mathfrak{g}$ by $V$. Further we choose a basis $\{x^{1},...,x^{d}\}$ of V, with $c^{k}_{ij}$ denoting the structure constants with respect to this basis, so by definition we have $[x^{i}x{j}]=c^{k}_{ij}x^{k}$, note that we will be assuming the summation convention in this subsection. The Hochschild cohomology of $U\mathfrak{g}$ is computed as Lie algebra cohomology with coefficients in the module $(U\mathfrak{g})^{ad}$. This can equivalently be described as Lie algebra cohomology of $\mathfrak{g}$ with coefficients in the module $S(\mathfrak{g}^{ad})$. We can restate this as the following lemma; \begin{lemma} The Hochschild cohomology of $U\mathfrak{g}$ is equivalent as a cdga to $$\Big(\mathbb{C}[x^{1},...,x^{d},\psi^{1},...,\psi^{d}],\partial:=c^{k}_{ij}x^{k}\psi^{j}\partial_{x^{i}}-\frac{1}{2}c^{i}_{jk}\psi^{j}\psi^{k}\partial_{\psi^{i}}\Big),$$ where the $x$ variables are of cohomological degree $0$ and the $\psi$ ones of cohomological degree $1$.\end{lemma} We'll construct $H^{*}_{ch}(U\mathfrak{g})$ as a BRST reduction of the vertex algebra of chiral polyvector fields on the affine space dual to the underlying vector space of $\mathfrak{g}$, $\Theta_{\mathbb{A}^{1}}^{ch}$. Considering the above lemma, it is pretty obvious how to do this. \begin{tcolorbox}\begin{definition}Define the complex $H^{*}_{ch}(U\mathfrak{g})$ as the BRST reduction $\Theta^{ch,BRST}_{V^{*},\omega_{\mathfrak{g}}}$. Here $$\omega_{\mathfrak{g}}:=c^{k}_{ij}x^{k}_{0}\psi^{j}_{0}y^{i}_{1}-\frac{1}{2}c^{i}_{jk}\psi^{j}_{0}\psi^{k}_{0}\phi^{i}_{1},$$ which we note satisfies all the requirements needed to perform the BRST reduction. \end{definition}\end{tcolorbox}

The reader can check that similar formulaes define a module, $H^{ch}_{*}(U\mathfrak{g})$, for $H^{*}_{ch}(U\mathfrak{g})$. The definition of the $S^{1}$ action is unchanged from the case of trivial Lie bracket.\begin{theorem} ($H^{*}_{ch}(U\mathfrak{g})$,$H^{ch}_{*}(U\mathfrak{g}))$ satisfies the conditions of the conjecture. \end{theorem}
\begin{proof} There is very little to check. We mention only that (5) can be proven by reducing to the case of an abelian $\mathfrak{g}$ after endowing the objects in question with the evident PBW fitrations. \end{proof}

\begin{remark} We expect such enhancements to exist in in the case of perfect complexes on an orbifold, further such objects oughtto be related to the orbifold chiral de Rham complex of Frenkel and Szczesny, \cite{FS}. We wish to return to these questions in the future.\end{remark}

\end{document}